\documentclass{amsart}
\usepackage{color}
\usepackage{amsmath}
\usepackage{amssymb}
\usepackage{amsfonts}
\usepackage{amsthm,url}
\usepackage{enumerate}
\usepackage{epsfig}
\usepackage[all]{xy}
\usepackage[utf8]{inputenc}
\usepackage[usenames,dvipsnames]{xcolor}
\usepackage{tkz-graph}
\usetikzlibrary{arrows}

\newcommand{\mb}{\mathbb}

\newcommand{\ol}{\overline}

\newcommand{\tr}{\,\text{tr}}

\newcommand{\C}{\mb{C}}

\newcommand{\bs}{\backslash}

\begin{document}
\newtheorem{Theo}{Theorem}
\newtheorem{Ques}{Question}
\newtheorem{Prop}[Theo]{Proposition}
\newtheorem{Exam}[Theo]{Example}
\newtheorem{Lemma}[Theo]{Lemma}

\newtheorem{Claim}{Claim}
\newtheorem{Cor}[Theo]{Corollary}
\newtheorem{Conj}[Theo]{Conjecture}

\theoremstyle{definition}
\newtheorem{Defn}[Theo]{Definition}
\newtheorem{Remark}[Theo]{Remark}
\title{Spectra of Coronae}

\let\thefootnote\relax\footnotetext{\noindent The published version of this
  paper appears in \emph{Linear Algebra and its Applications}, Volume
  435, no.~5, (2011), and contains occasional simplifications and additions
  beyond the current version.}

\author{Cam McLeman and Erin McNicholas}

\begin{abstract}
We introduce a new invariant, the \emph{coronal} of a graph, and use
it to compute the spectrum of the corona $G\circ H$ of two graphs $G$
and $H$.  In particular, we show that this spectrum is completely
determined by the spectra of $G$ and $H$ and the coronal of $H$.
Previous work has computed the spectrum of a corona only in the case
that $H$ is regular.  We then explicitly compute the coronals for
several families of graphs, including regular graphs, complete
$n$-partite graphs, and paths.  Finally, we use the corona
construction to generate many infinite families of pairs of cospectral
graphs.
\end{abstract}

\maketitle
\section{Introduction}
Let $G$ and $H$ be (finite, simple, non-empty) graphs.  The
\emph{corona} $G\circ H$ of $G$ and $H$ is constructed as follows:
Choose a labeling of the vertices of $G$ with labels $1,2,\ldots,m$.
Take one copy of $G$ and $m$ disjoint copies of $H$, labeled
$H_1,\ldots,H_m$, and connect each vertex of $H_i$ to vertex $i$ of
$G$.  This construction was introduced by Frucht and Harary
\cite{FH70} with the (achieved) goal of constructing a graph whose
automorphism group is the wreath product of the two component
automorphism groups.  Since then, a variety of papers have appeared
investigating a wide range of graph-theoretic properties of coronas,
such as the bandwidth \cite{CLY92}, the minimum sum \cite{Wi93},
applications to Ramsey theory \cite{Ne85}, etc.  Further, the spectral
properties of coronas are significant in the study of invertible
graphs.  Briefly, a graph $G$ is invertible if the inverse of the
graph's adjacency matrix is diagonally similar to the adjacency matrix
of another graph $G^+$.  Motivated by applications to quantum
chemistry, Godsil \cite{Go85} studies invertible bipartite graphs with
a unique perfect matching.  In response to his question asking for a
characterization of such graphs with the additional property that
$G\cong G^+$, Simian and Cao \cite{SC89} determine the answer to be
exactly the coronas of bipartite graphs with the single-vertex graph
$K_1$.

\vspace*{.1in}

The study of spectral properties of coronas was continued by Barik,
et. al. in \cite{BPS07}, who found the spectrum of the corona $G\circ
H$ in the special case that $H$ is regular.  In Section 2, we drop the
regularity condition on $H$ and compute the spectrum of the corona of
any pair of graphs using a new graph invariant called the coronal. In
Section 3, we compute the coronal for several families of graphs,
including regular graphs examined in \cite{BPS07}, complete
$n$-partite graphs, and path graphs.  Finally, in Section 4, we see
that properties of the spectrum of coronas lends itself to
finding many large families of cospectral graph pairs.

\subsection*{Notation}
The symbols ${\bf 0}_n$ and ${\bf 1}_n$ (resp., ${\bf 0}_{mn}$ and
${\bf 1}_{mn}$) will stand for length-$n$ column vectors (resp.
$m\times n$ matrices) consisting entirely of 0's and 1's.  For two
matrices $A$ and $B$, the matrix $A\otimes B$ is the Kronecker (or
tensor) product of $A$ and $B$.  For a graph $G$ with adjacency matrix
$A$, the characteristic polynomial of $G$ is
$f_G(\lambda):=\det(\lambda I-A)$.  We use the standard notations
$P_n$, $C_n$, $S_n$, and $K_n$ respectively for the path, cycle, star,
and complete graphs on $n$ vertices.

\section{The Main Theorem}
Let $G$ and $H$ be finite simple graphs on $m$ and $n$ vertices,
respectively, and let $A$ and $B$ denote their respective adjacency
matrices.  We begin by choosing a convenient labeling of the vertices
of $G\circ H$.  Recall that $G\circ H$ is comprised of the $m$
vertices of $G$, which we label arbitrarily using the symbols
$\{1,2,\ldots,m\}$, and $m$ copies $H_1,H_2,\ldots,H_m$ of $H$.
Choose an arbitrary ordering $h_1,h_2,\ldots,h_n$ of the vertices of
$H$, and label the vertex in $H_i$ corresponding to $h_k$ by the label
$i+mk$.  Below is a sample corona with the above labeling procedure:
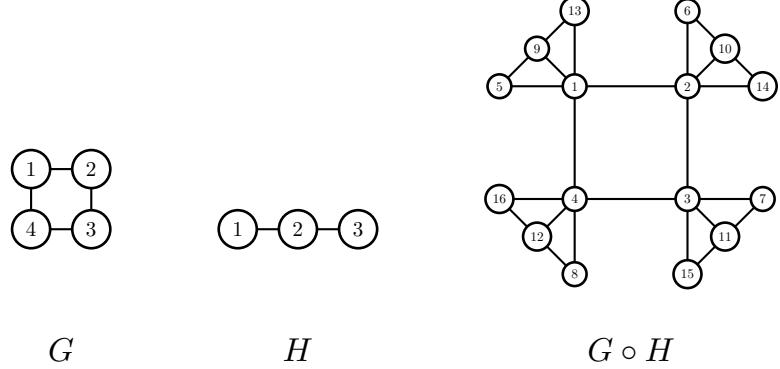
\begin{center}
\begin{tabular}{ccc}
\quad\quad\begin{tikzpicture}[scale=0.8,transform shape]
  \Vertex[x=0,y=1]{1}
  \Vertex[x=1,y=1]{2}
  \Vertex[x=1,y=0]{3}
  \Vertex[x=0,y=0]{4}
  \Edge(1)(2)
  \Edge(2)(3)
  \Edge(3)(4)
  \Edge(4)(1)
  \draw(.5,-2) node[scale=1.5625] {$G$};
\end{tikzpicture}\quad\quad
&
\quad\quad\begin{tikzpicture}[scale=0.8,transform shape]
  \Vertex[x=0,y=0]{1}
  \Vertex[x=1,y=0]{2}
  \Vertex[x=2,y=0]{3}
  \Edge(1)(2)
  \Edge(2)(3)
  \draw(1,-2) node[scale=1.5625] {$H$};
\end{tikzpicture}\quad\quad
&
\quad\quad\begin{tikzpicture}[scale=0.5,transform shape]
  \Vertex[x=0,y=3]{1}
  \Vertex[x=-2,y=3]{5}
  \Vertex[x=-1,y=4]{9}
  \Vertex[x=0,y=5]{13}

  \Vertex[x=3,y=3]{2}
  \Vertex[x=3,y=5]{6}
  \Vertex[x=4,y=4]{10}
  \Vertex[x=5,y=3]{14}

  \Vertex[x=3,y=0]{3}
  \Vertex[x=5,y=0]{7}
  \Vertex[x=4,y=-1]{11}
  \Vertex[x=3,y=-2]{15}

  \Vertex[x=0,y=0]{4}
  \Vertex[x=0,y=-2]{8}
  \Vertex[x=-1,y=-1]{12}
  \Vertex[x=-2,y=0]{16}

  \tikzstyle{LabelStyle}=[fill=white,sloped]
  \foreach \x/\y in {1/2,2/3,3/4,4/1}{\Edge(\x)(\y)}
  \foreach \x/\y in {1/5,1/9,1/13,5/9,9/13}{\Edge(\x)(\y)}
  \foreach \x/\y in {2/6,2/10,2/14,6/10,10/14}{\Edge(\x)(\y)}
  \foreach \x/\y in {3/7,3/11,3/15,7/11,11/15}{\Edge(\x)(\y)}
  \foreach \x/\y in {4/8,4/12,4/16,8/12,12/16}{\Edge(\x)(\y)}
  \draw(1.5,-4) node [scale=2.5]{$G\circ H$};
\end{tikzpicture}
\end{tabular}
\end{center}
\noindent Under this labelling the
adjacency matrix of $G\circ H$ is given by
$$
A\circ B:=\begin{bmatrix}
A&{\bf 1}_n^T\otimes I_m\\
{\bf 1}_n\otimes I_m&B\otimes I_m
\end{bmatrix}.
$$ The goal now is to compute the eigenvalues of this corona matrix in
terms of the spectra of $A$ and $B$.  We introduce one new invariant
for this purpose.

\begin{Defn}
Let $H$ be a graph on $n$ vertices, with adjacency matrix $B$. Note
that, viewed as a matrix over the field of rational functions
$\C(\lambda)$, the characteristic matrix $\lambda I-B$ has determinant
$\det(\lambda I-B)=f_H(\lambda)\neq 0$, so is invertible.  The
\emph{coronal} $\chi_H(\lambda)\in \C(\lambda)$ of $H$ is defined to
be the sum of the entries of the matrix $(\lambda I-B)^{-1}$.  Note
this can be calculated as
$$
\chi_H(\lambda)={\bf 1}_n^T(\lambda I_n-B)^{-1}{\bf 1}_n.
$$
\end{Defn}

Our main theorem is that, beyond the spectra of $G$ and $H$, only the
coronal of $H$ is needed to compute the spectrum of $G\circ H$.

\begin{Theo}\label{BigThm}
Let $G$ and $H$ be graphs with $m$ and $n$ vertices.  Let
$\chi_H(\lambda)$ be the coronal of $H$.  Then the characteristic
polynomial of $G\circ H$ is
$$f_{G\circ H}(\lambda)=f_H(\lambda)^mf_G(\lambda-\chi_H(\lambda)).$$
In particular, the spectrum of $G\circ H$ is completely determined by
the characteristic polynomials $f_G$ and $f_H$, and the coronal
$\chi_H$ of $H$.
\end{Theo}
\begin{proof}
\noindent Let $A$ and $B$ denote the respective adjacency matrices
of $G$ and $H$.  We compute the characteristic polynomial of the
matrix $A\circ B$.  For this, we recall two elementary results from
linear algebra on the multiplication of Kronecker products and
determinants of block matrices:
\begin{itemize}
\item In cases where each multiplication makes sense, we have
$$M_1M_2\otimes M_3M_4=(M_1\otimes M_3)(M_2\otimes M_4).$$
\item If $M_4$ is invertible, then
$$
\det\begin{pmatrix}M_1&M_2\\M_3&M_4\end{pmatrix}=\det(M_4)\det(M_1-M_2M_4^{-1}M_3).
$$
\end{itemize}
Combining these two facts, we have (as an equality of rational functions)
\begin{align*}
f_{G\circ H}(\lambda)&=\det(\lambda I_{m(n+1)}-A\circ B)\\
&=\det\begin{pmatrix}\lambda I_m-A&-{\bf 1}_n^T\otimes I_m\\-{\bf 1}_n\otimes I_m&\lambda I_{mn}-B\otimes I_m\end{pmatrix}\\
&=\det\begin{pmatrix}\lambda I_m-A&-{\bf 1}_n^T\otimes I_m\\-{\bf 1}_n\otimes I_m&(\lambda I_n-B)\otimes I_m\end{pmatrix}\\
&=\det((\lambda I_n-B)\otimes I_m)\det\bigg[(\lambda I_m-A)-({\bf 1}_n^T\otimes I_m)((\lambda I_n-B)\otimes I_m)^{-1}({\bf 1}_n\otimes I_m)\bigg]\\
&=\det(\lambda I_n-B)^m\det(\lambda I_m-A-({\bf 1}_n^T(\lambda I_n-B)^{-1}{\bf 1}_n)\otimes I_m)\\
&=\det(\lambda I_n-B)^m\det((\lambda-\chi_H(\lambda))I_m-A)\\
&=f_H(\lambda)^mf_G(\lambda-\chi_H(\lambda)).
\end{align*}
\end{proof}

\begin{Remark}\label{cospec}
A natural question is whether or not the spectrum of $G\circ H$ is
determined by the spectra of $G$ and $H$, i.e., whether knowledge of
the coronal is strictly necessary.  Indeed it is: Computing the
coronals of the cospectral graphs $S_5$ and $C_4\cup K_1$, we have
\[
\chi_{S_5}(\lambda)=\frac{5\lambda+8}{\lambda^2-4}\quad\quad\text{ and }\quad\quad \chi_{C_4\cup K_1}(\lambda)=\frac{5\lambda-2}{\lambda^2-2\lambda}.
\] Thus cospectral graphs need not have the same coronal, and hence
for a given graph $G$, the spectra of $G\circ S_5$ and $G\circ(C_4\cup
K_1)$ are almost always distinct.  Note that this stands in
stark contrast to the situation for the Cartesian and tensor products
of graphs.  In both of these cases, the spectrum of the product is
determined by the spectra of the components.
\end{Remark}

The unexpected simplicity of the examples in the above remark are
representative of a fairly common phenomena: since the coronal
$\chi_H(\lambda)=\frac{\widetilde{\chi}_H(\lambda)}{f_H(\lambda)}$ can
be computed as the quotient of the sum $\widetilde{\chi}_H(\lambda)$
of the cofactors of $\lambda I-B$ by the characteristic polynomial
$f_H(\lambda)$, it is \emph{a priori} the quotient of a degree $n-1$
polynomial by a degree $n$ polynomial.  In practice, however, as in
the examples in the remark, these two polynomials typically have roots
in common, providing for a reduced expression for the coronal.  Let us
suppose that
$g(\lambda):=\gcd(\widetilde{\chi}_H(\lambda),f_H(\lambda))$ has
degree $n-d$ (the $\gcd$ being taken in $\C[\lambda]$), so that
$\chi_H(\lambda)$ in its reduced form is a quotient of a degree $d-1$
polynomial by a degree $d$ polynomial.  Since the denominator of this
reduced fraction is a factor of $f_H(\lambda)$, and since $f_G$ is of
degree $m$, each pole of $\chi_H(\lambda)$ is simultaneously a
multiplicity-$m$ pole of $f_G(\lambda-\chi_H(\lambda))$ and a
multiplicity-$m$ root of $f_H(\lambda)^m$.  Since these contributions
cancel in the overall determination of the roots of $f_{G\circ
  H}(\lambda)$ in the expression

$$f_{G\circ H}(\lambda)=f_H(\lambda)^mf_G(\lambda-\chi_H(\lambda))$$

\noindent from Theorem \ref{BigThm}, we can now more explicitly
describe the spectrum of the corona.  Namely, let $d$ be the degree of
the denominator of $\chi_H(\lambda)$ as a reduced fraction.  Then the
spectrum of $G\circ H$ consists of:
\begin{itemize}
\item Some ``old'' eigenvalues, i.e., the roots of $f_H(\lambda)$
  which are not poles of $\chi_H(\lambda)$ (or equivalently, the roots
  of $g(\lambda)$), each with multiplicity $|G|$; and
\item Some ``new'' eigenvalues, i.e., the values $\lambda$ such that
  $\lambda-\chi_H(\lambda)$ is an eigenvalue $\mu$ of $G$ (with the
  multiplicity of $\lambda$ equal to the multiplicity of $\mu$ as an
  eigenvalue of $G$.)
\end{itemize}
Since for a given $\mu$, solving $\lambda-\chi_H(\lambda)=\mu$ by
clearing denominators amounts to finding the roots of a degree $d+1$
polynomial in $\lambda$, the above two bullets combine to respectively
provide all $(n-d)m+m(d+1)=m(n+1)$ eigenvalues of $G\circ H$.

\bigskip

The following table, computed using SAGE (\cite{sage}), gives the
number of graphs on $n$ vertices whose coronal has a denominator of
degree $d$ (as a reduced fraction), as well as the average degree of
this denominator, for $1\leq n\leq 7$.

\begin{table}[!ht]
\centering
\parbox{195pt}{Table 1:  Number of graphs on $n$ vertices whose coronal has denominator of degree $d$}
\begin{tabular}{|c|c|c|c|c|c|c|c|}\hline
$d\bs n$&1&2&3&4&5&6&7\\\hline
1&1&2&2&4&3&8&6\\
2&&0&2&5&12&28&44\\
3&&&0&2&13&50&138\\
4&&&&0&6&40&304\\
5&&&&&0&22&246\\
6&&&&&&8&214\\
7&&&&&&&92\\\hline
Total&1&2&4&11&34&156&1044\\\hline
Average $d$&1&1&1.5&1.82&2.65&3.41&4.68\\\hline
(Average $d$)/$n$&1&0.5&0.5&0.45&0.53&0.57&0.66\\\hline
\end{tabular}
\end{table}

\vspace*{.1in}

Since determining the characteristic polynomial of $G\circ H$ from the
spectra of $G$ and $H$ requires only the extra knowledge of the
coronal of $H$, it remains to develop techniques for computing these
coronals.  In Section 3, we will develop shortcuts for these
computations, but we briefly conclude this section with some more
computationally-oriented approaches.  A first such option is to have a
software package with linear algebra capabilities directly compute the
inverse of $\lambda I-B$ and sum its entries, as done in the
computations for Table 1.  This seems to be computationally feasible
only for rather small graphs (e.g., $n\leq 12$).  A second, more
graph-theoretic, option relies on a combinatorial result of Schwenk
\cite{Sc91} to compute each cofactor of $\lambda I-B$ individually,
before summing them to compute the coronal:

\begin{Theo}[Schwenk, \cite{Sc91}]\label{Schwenk}
For vertices $i$ and $j$ of a graph $H$ with adjacency matrix $B$, let
$\mathcal{P}_{i,j}$ denote the set of paths from $i$ to $j$.  Then
$$
\operatorname{adj}(\lambda I-B)_{i,j}=\sum_{P\in \mathcal{P}_{i,j}}f_{H-P}(\lambda).
$$
\end{Theo}
\noindent Again, this approach becomes computationally infeasible
fairly quickly without a method for pruning the number of cofactors to
calculate.  We explore this idea in the next section.  Regardless,
from Theorem \ref{Schwenk}, we obtain:
\begin{Cor}
The spectrum of the corona $G\circ H$ is determined by the spectrum of
$G$ and the spectra of the proper subgraphs of $H$ (or more
economically, only from the spectra of those subgraphs obtained by
deleting paths from $H$).
\end{Cor}

\section{Computing Coronals}\label{exs}
\noindent In this section, we will compute the coronals
$\chi_H(\lambda)$ for several families of graphs, and hence for such
$H$ obtain the full spectrum of the corona $G\circ H$ for any $G$.
The principal technique exploits the regularity or near-regularity of
a graph in order to greatly reduce the number of cofactor calculations
(relative to those required by Theorem \ref{Schwenk}) needed to
compute the coronal.  In particular, we use these ideas to compute the
coronals of regular graphs, complete bipartite graphs, and paths.

For graphs that are ``nearly regular'' in the sense that their degree
sequences are almost constant, we can take advantage of
linear-algebraic symmetries to compute the coronals.  We begin with
two concrete computations, those corresponding to regular and complete
bipartite graphs, before extracting the underlying heuristic and
applying it to the coronal of path graphs.  The case of regular
graphs, first addressed in \cite{BPS07}, is particularly
straight-forward from this viewpoint.

\begin{Prop}[Regular Graphs]\label{regular}
Let $H$ be $r$-regular on $n$ vertices.  Then
$$\chi_H(\lambda)=\frac{n}{\lambda-r}.$$ Thus for any graph $G$, the
spectrum of $G\circ H$ is precisely:
\begin{itemize}
\item Every non-maximal eigenvalue of $H$, each with multiplicity $|G|$.
\item The two eigenvalues
$$
\frac{\mu+r\pm\sqrt{(r-\mu)^2+4n}}{2}
$$
for each eigenvalue $\mu$ of $G$.
\end{itemize}
\end{Prop}
\begin{proof}
Let $B$ be the adjacency matrix of $H$.  By regularity, we have $B{\bf
  1}_n=r{\bf 1}_n$, and hence $(\lambda I-B){\bf 1}_n=(\lambda-r){\bf 1}_n$.
Cross-dividing and multiplying by ${\bf 1}_n^T$,
$$\chi_H(\lambda)={\bf 1}_n^T(\lambda I-B)^{-1}{\bf 1}_n=\frac{{\bf
    1}_n^T{\bf 1}_n}{\lambda-r}=\frac{n}{\lambda-r}.$$ The only pole
of $\chi_H(\lambda)$ is the maximal eigenvalue $\lambda=r$ of $H$, and
the ``new'' eigenvalues are obtained by solving
$\lambda-\frac{n}{\lambda-r}=\mu$ for each eigenvalue $\mu$ of $G$.
\end{proof}
It is noteworthy that all $r$-regular graphs on $n$ vertices have the
same coronal, especially given that the cofactors of the corresponding
matrices $(B-\lambda I)^{-1}$ appear to be markedly dissimilar.  The
simplicity of this scenario, and the easily checked observation that
cospectral regular graphs must have the same regularity, lead to the
following corollary.  We will make use of this corollary in the final
section.

\begin{Cor}\label{reg}
Cospectral regular graphs have the same coronal.
\end{Cor}

\noindent As a second class of examples, we compute the coronals of
complete bipartite and $n$-partite graphs.

\begin{Prop}[Complete Bipartite Graphs]\label{cbg}
Let $H=K_{p,q}$ be a complete bipartite graph on $p+q=n$ vertices.
Then
$$\chi_H(\lambda)=\frac{n\lambda+2pq}{\lambda^2-pq}.$$
For any graph $G$, the spectrum of $G\circ H$ is given by:
\begin{itemize}
\item The eigenvalue $0$ with multiplicity $m(n-2)$; and
\item For each eigenvalue $\mu$ of $G$, the roots of the polynomial
$$x^3-\mu x^2-(p+q+pq)x+pq(\mu-2).$$
\end{itemize}
\end{Prop}
\begin{proof}
Let $B=\begin{bmatrix}{\bf 0}_{p,p}&{\bf 1}_{p,q}\\{\bf 1}_{q,p}&{\bf
  0}_{q,q}\end{bmatrix}$ be the adjacency matrix of $K_{p,q}$ and let
$X=\text{diag}((q+\lambda)I_p,(p+\lambda)I_q)$ be the diagonal matrix
with the first $p$ diagonal entries being $(q+\lambda)$ and the last
$q$ entries being $(p+\lambda)$.  Then $(\lambda I-B)X{\bf
  1}_n=(\lambda^2-pq){\bf 1}_n$, and so
$$
\chi_H(\lambda)={\bf 1}_n^T(\lambda I-B)^{-1}{\bf 1}_n=\frac{{\bf 1}_n^TX{\bf 1}_n}{\lambda^2-pq}=\frac{(p+q)\lambda+2pq}{\lambda^2-pq}.
$$
\noindent
Thus the coronal has poles at both of the non-zero eigenvalues
$\pm\sqrt{pq}$ of $K_{p,q}$, leaving only the eigenvalue 0 with
multiplicity $p+q-2$.  Finally, solving $\lambda-\chi_H(\lambda)=\mu$
gives the new eigenvalues in the spectrum as stated in the
proposition.
\end{proof}
\begin{Remark}
It might be tempting in light of Propositions \ref{regular} and
\ref{cbg} to hope that the degree sequence of a graph determines its
coronal.  This too, like the analogous conjecture stemming from
cospectrality (Remark \ref{cospec}), turns out to be false: The graphs
$P_5$ and $K_2\cup K_3$ have the same degree sequence, but we find by
direct computation that
$$
\chi_{P_5}(\lambda)=\frac{5\lambda^2+8\lambda-1}{\lambda^3-3\lambda}\quad\quad\quad\quad\chi_{K_2\cup K_3}(\lambda)=\frac{5\lambda-7}{\lambda^2-3\lambda+2}.
$$
\end{Remark}

\noindent Similar to the complete bipartite computation, we have the
following somewhat technical generalization to complete $k$-bipartite
graphs.
\begin{Prop}
Let $H$ be the complete $k$-partite graph $K_{n_1,n_2,\ldots,n_k}$.
Then
$$\chi_H(\lambda)=\left(\frac{\prod_{j=1}^k(n_j+\lambda)}{\sum_{j=1}^kn_j\prod\limits^k_{\substack{{i=1}\\i\neq
j}}(n_i+\lambda)}-1\right)^{-1}=\frac{\sum_{j=1}^kjC_j\lambda^{k-j}}{\lambda^k-\sum_{j=2}^k(j-1)C_j\lambda^{k-j}}$$
where $C_j$ is the sum of the ${k \choose j}$ products of the form
$n_{i_1}n_{i_2}\cdots n_{i_j}$ with distinct indices.
\end{Prop}
\begin{proof}
Let $B$ be the adjacency matrix of $H$, let
$g_j(\lambda)=\prod\limits^k_{\substack{{i=1}\\i\neq
    j}}(n_i+\lambda)$, and let $X$ be the block diagonal matrix
whose $j$-th block is $g_jI_{n_j}$.  Then
$$(\lambda I_n-B)X{\bf 1}_n=\left[\prod_{i=1}^k(n_i+\lambda)-\sum_{j=1}^kn_jg_j\right]{\bf 1}_n.$$
Solving as in the bipartite case, we find
\begin{align*}
\chi_H(\lambda)={\bf 1}_n^T(\lambda I-B)^{-1}{\bf 1}_n&=\frac{{\bf
1}_n^TX{\bf 1}_n}{\prod_{i=1}^k(n_i+\lambda)-\sum_{j=1}^kn_jg_j}\\
&=\frac{\sum_{j=1}^kn_jg_j}{\prod_{i=1}^k(n_i+\lambda)-\sum_{j=1}^kn_jg_j},
\end{align*}
which gives the result.
\end{proof}

The proof technique for the last two propositions generalizes to
``nearly regular'' graphs, by which we mean graphs $H$ for which all but a small
number of vertices have the same degree $r$.  In this case, we can write
$$(\lambda I-B){\bf 1}_n=(\lambda-r){\bf 1}_n+{\bf v},$$ where ${\bf
  v}=(v_i)$ is a vector consisting mostly of 0's.  This gives
$$(\lambda I-B)^{-1}{\bf 1}_n=\frac{1}{\lambda-r}\left[{\bf 1}_n-(\lambda I-B)^{-1}{\bf v}\right],$$
and thus, using the adjugate formula for the determinant,
$$\chi_H(\lambda)={\bf 1}_n^T(\lambda I-B)^{-1}{\bf
  1}_n=\frac{1}{\lambda-r}\left[n-\frac{1}{f_H(\lambda)}\sum_{1\leq
    i,j\leq n}v_iC_{i,j}\right],$$ where $C_{i,j}$ denotes the
$(i,j)$-cofactor of $\lambda I-B$.  Since $v_i$ is zero for most
values of $i$, we have an effective technique for computing coronals
if we can compute a small number of cofactors (as opposed to, in
particular, computing \emph{all} of the cofactors and using Theorem
\ref{Schwenk}). For example, if we let $f_n=f_{P_n}(\lambda)$ be the
characteristic polynomial of the path graph $P_n$ on $n$ vertices (by
convention, set $f_0=1$), we can compute coronals of paths as
follows:

\begin{Prop}[Path Graphs]\label{path}
Let $H=P_n$.  Then
$$\chi_H=\frac{nf_n-2\sum_{j=0}^{n-1}f_j}{(\lambda-2)f_n}.$$
\end{Prop}
\begin{proof}
In the notation of the discussion preceding the proposition, we take
$r=2$ and ${\bf v}=[1\, 0\, 0\, \cdots\, 0\, 0\, 1]^T$.  Further, we
note that an easy induction argument using cofactor expansion gives
$C_{j,1}=C_{j,n}=f_{j-1}$.  Thus we obtain
$$\chi_{P_n}(\lambda)=\frac{1}{\lambda-2}\left[n-\frac{1}{f_n}\sum_{j=1}^n
  (C_{j,1}+C_{j,n})\right]=\frac{1}{\lambda-2}\left[n-\frac{2}{f_n}\sum_{j=0}^{n-1}
  f_j\right],$$
from which the result follows.
\end{proof}
\noindent From this, we easily calculate the coronals for the first few path graphs:

\begin{table}[!ht]
\centering
\parbox{195pt}{Table 2: The coronals $\chi_{P_n}(\lambda)$ for $1\leq n\leq 7$}

\vspace*{.1in}

\begin{tabular}{c|ccccccc}
$n$&1&2&3&4&5&6&7\\\hline
$\chi_{P_n}(\lambda)$&$\frac{1}{\lambda}$&$\frac{2}{\lambda-1}$&$\frac{3\lambda+4}{\lambda^2-2}$&$\frac{4\lambda+2}{\lambda^2-\lambda-1}$&$\frac{5\lambda^2+8\lambda-1}{\lambda^3-3\lambda}$&$\frac{6\lambda^2
+ 4\lambda - 4}{\lambda^3 - \lambda^2 - 2\lambda +
1}$&$\frac{7\lambda^3 + 12\lambda^2 - 6\lambda - 8}{\lambda^4 -
4\lambda^2 + 2}$
\end{tabular}
\end{table}

\begin{Remark}
This particular example can also be computed using Theorem \ref{Schwenk}: For $i$ and $j$
distinct, there is a unique path $[i,j]$ from vertex $i$ to vertex $j$, so the
sum in the theorem reduces to a single term:
$$
\operatorname{adj}(\lambda I-B)_{ij}=f_{P_n-[i,j]}=f_{i-1}f_{n-j}.
$$
Similarly, we find $\operatorname{adj}(\lambda I-B)_{ii}=f_{i-1}f_{n-i}$.
Summing over all the cofactors gives
$$
\chi_{P_n}(\lambda)=\frac{1}{f_n}\left(\sum_{i=1}^nf_{i-1}f_{n-i}+2\sum_{i,j=1}^nf_{i-1}f_{n-j}\right),
$$ which reduces to the result computed in Proposition \ref{path}
after some arithmetic.
\end{Remark}

\section{Cospectrality}

At the end of \cite{BPS07}, the authors note that if $G_1$ and $G_2$
are cospectral graphs, then $G_1\circ K_1$ and $G_2\circ K_1$ are also
cospectral, and that (by repeated coronation with $K_1$) this leads to
an infinite collection of cospectral pairs.  Armed with the
characteristic polynomial
$$f_{G\circ H}(\lambda)=f_H(\lambda)^mf_G(\lambda-\chi_H(\lambda))$$
of the corona (Theorem \ref{BigThm}), we can greatly generalize this
observation on two fronts.

\begin{Cor}\label{cospec2}
If $G_1$ and $G_2$ are cospectral, and $H$ is any graph, then
$G_1\circ H$ and $G_2\circ H$ are cospectral.  Further, if $H_1$ and
$H_2$ are cospectral and $\chi_{H_1}=\chi_{H_2}$, and $G$ is any
graph, then $G\circ H_1$ and $G\circ H_2$ are cospectral.
\end{Cor}
We remark that examples of this second type do indeed exist.  
Define the \emph{switching graph} $\text{Sw}(T)$ of a tree $T$ with adjacency matrix
$A_T$ to be the graph with adjacency matrix
\[
A_{\text{Sw}(T)}:=\begin{bmatrix}1&0\\0&1\end{bmatrix}\otimes A_T+\begin{bmatrix}0&1\\1&0\end{bmatrix}\otimes A_{\ol{T}},
\]
and let $T_1$ and $T_2$ be non-isomorphic cospectral trees with
cospectral complements (note that by \cite{GM76}, generalizing
\cite{Sc73}, ``almost all'' trees admit a cospectral pair with
cospectral complement).  Then the switching graphs $\text{Sw}(T_1)$
and $\text{Sw}(T_2)$ are non-isomorphic cospectral regular graphs (see
\cite{GM82}, Construction 3.7), and also have the same coronal by
Corollary \ref{reg}.  Corollary \ref{cospec2} now implies that for
\emph{any} graphs $G$ and $H$, we have the cospectral pair $G\,\circ
\,\text{Sw}(T_1)$ and $G\,\circ \,\text{Sw}(T_2)$ and the cospectral
pair $\text{Sw}(T_1)\circ H$ and $\text{Sw}(T_2)\circ H$.  This gives,
for example, infinitely many cospectral pairs of graphs with any given
graph $G$ as an induced subgraph.

\end{document}